\documentclass[a4paper]{amsart}
\usepackage{amsmath,amssymb}
\usepackage[dvipdfmx]{graphicx}
\usepackage{amsthm}
\usepackage{bm}
\usepackage{tikz}
\usepackage{ascmac}
\usepackage{mathrsfs}
\usepackage[noadjust]{cite}
\usepackage{ytableau}
\usepackage{braket} 
\usepackage{mathtools}\mathtoolsset{showonlyrefs=true}
\usepackage{adjustbox}

\usepackage{newtxtext}
\usepackage{bbold}
\usepackage{stmaryrd}
\usepackage{setspace}

\newtheorem{theorem}{Theorem}[section]
\newtheorem*{main theorem}{Main Theorem}

\newtheorem{lemma}{Lemma}[section]

\newtheorem{proposition}{Proposition}[section]

\theoremstyle{definition}

\theoremstyle{remark}

\numberwithin{equation}{section}

\def\us#1_#2{\underset{#2}{#1}}
\def\os#1^#2{\overset{#2}{#1}}
\DeclareFontFamily{U}{wncy}{}
\DeclareFontShape{U}{wncy}{m}{n}{<->wncyr10}{}
\DeclareSymbolFont{mcy}{U}{wncy}{m}{n}
\DeclareMathSymbol{\Sh}{\mathord}{mcy}{`X}
\DeclareMathSymbol{\sh}{\mathord}{mcy}{`x}
\renewcommand{\boldsymbol}[1]{\text{\bf \textit{#1}}}
\newcommand{\mysh}{\widetilde{\sh}}

\newcommand{\OT}[1]{\{1,3\}^{#1}}
\newcommand{\TO}[1]{\{3,1\}^{#1}}

\newcommand{\HS}{\widehat{\mathcal{S}}}
\newcommand{\HU}{\widehat{\mathcal{U}}}

\title[$t$-adic SMZVs for indices with alternating $1$ and $3$, starting with $1$ and ending with $3$]
{$t$-adic symmetric multiple zeta values for indices with alternating $1$ and $3$, starting with $1$ and ending with $3$}

\author{Kento Fujita}
\address[Kento Fujita]
{ARISE analytics, inc.,
2-21-1 Shibuya, Shibuya-ku, Tokyo 150-0002, Japan}
\email{kento.fujita@al.rikkyo.ac.jp}

\calclayout

\begin{document}
\maketitle
\begin{abstract}
Hirose, Murahara, and Saito proved that 
some $t$-adic symmetric multiple zeta values, for indices in which $1$ and $3$ appear alternately in succession,
can be expressed as polynomials in Riemann zeta values, and conjectured similar formulas.
In this paper, we prove the conjectured formula for indices that start with $1$ and end with $3$,
showing that they also can be expressed as polynomials in Riemann zeta values.
\end{abstract}
\section{Introduction}
An index is a sequence of positive integers, including the empty sequence $\emptyset$.
% The length of an index $\boldsymbol{k}$ is called its depth. 
We call $\boldsymbol{k}$ an admissible index if its last component is greater than $1$ or if $\boldsymbol{k}=\emptyset$.
For an admissible index $\boldsymbol{k}=(k_1,\ldots,k_r)$, the multiple zeta value (MZV) is defined by
\begin{equation}
    \zeta(\boldsymbol{k})
    =
    \zeta(k_1,\ldots,k_r)
    :=
    \sum_{0<n_1<\cdots<n_r}
    \frac{1}{n_1^{k_1}\cdots n_r^{k_r}}
    \in\mathbb{R},
\end{equation}
where $\zeta(\emptyset)$ is defined to be $1$.
% Note that depth is defined for an index, but it is often used for a MZV as well. 
Let $\mathcal{Z}$ be the $\mathbb{Q}$-algebra generated by all MZVs.
For an index $\boldsymbol{k}$, whether admissible or not, $\zeta^*(\boldsymbol{k})$ is defined as the constant term of 
the $*$-regularized MZV (see \cite{IKZ06} for details), which is a real number in $\mathcal{Z}$.
Note that $\zeta^*(1)=0$ and that $\zeta^*(\boldsymbol{k})=\zeta(\boldsymbol{k})$ if $\boldsymbol{k}$ is admissible.

For a nonnegative integer $m$ and an index $\boldsymbol{k}=(k_1,\ldots,k_r)$, set
\begin{equation}
    \zeta^*_m(\boldsymbol{k})
    =
    \zeta^*_m(k_1,\ldots,k_r)
    :=
    \sum_{\substack{
        l_1,\ldots,l_r\geq 0
        \\
        l_1+\cdots+l_r=m
    }}
    \zeta^*(k_1+l_1,\ldots,k_r+l_r)
    \prod_{i=1}^{r}
    \binom{k_i+l_i-1}{l_i}
    \in\mathbb{R}.
\end{equation}
For the empty index, $\zeta^*_m(\emptyset)$ is understood as $\delta_{m,0}$, 
where $\delta_{m,0}$ denotes the Kronecker delta.
The $t$-adic symmetric multiple zeta value ($t$-adic SMZV) is defined 
as an element in $\mathcal{Z}\llbracket t \rrbracket$ by
\begin{equation}
    \zeta^*_{\HS}(\boldsymbol{k})
    =
    \zeta^*_{\HS}(k_1,\ldots,k_r)
    :=
    \sum_{i=0}^{r}
    (-1)^{k_r+\cdots+k_{i+1}}
    \zeta^*(k_1,\ldots,k_i)
    \sum_{m=0}^\infty
    \zeta_m^*(k_r,\ldots,k_{i+1})t^m
    \in\mathcal{Z}\llbracket t \rrbracket.
\end{equation}
Here, $\zeta^*_{\HS}(\emptyset)$ is understood as $1$.
The $t$-adic SMZVs have been studied in \cite{Jar24,Ros19,OSY21,HMO21} as a generalization of 
the ordinary SMZVs (the $t$-adic SMZVs for $t=0$ or the $t$-adic SMZVs modulo $t$; see \cite{KZ}).
For every positive integer $m$, let $\pi_m:\mathcal{Z}\llbracket t \rrbracket\rightarrow \mathcal{Z}\llbracket t \rrbracket/t^m$ be the natural projection.
Then, $\mathcal{S}_m$-symmetric multiple zeta values ($\mathcal{S}_m$-SMZVs) are defined as follows:
\begin{equation}
    \zeta^*_{\mathcal{S}_m}(\boldsymbol{k}) 
    :=
    \pi_m(\zeta^*_{\HS}(\boldsymbol{k}))
    \in \mathcal{Z}\llbracket t \rrbracket/t^m.
\end{equation}
Hereafter, we may consider operations on $\mathcal{S}_m$-SMZVs modulo $\pi^2$.
Such elements should be understood as belonging to 
$(\mathcal{Z}/\pi^2\mathcal{Z})\llbracket t \rrbracket / t^m$.
Furthermore, throughout this paper, the empty sum is understood as zero.

It is known that MZVs can be expressed as polynomials in terms of Riemann zeta values (see \cite{BBBL01, BC20, BY18, BB03}).  
Based on this, Ono, Sakurada, and Seki investigated a similar structure for SMZVs and, in \cite[Theorem 4.1]{OSS21}, determined the values of $\mathcal{S}_2$-SMZVs modulo $\pi^2$ for a broader range of indices.  
Further extending this line of research, \cite{HMS24} examines the case where the $t$-adic SMZV for certain special indices, namely  
\begin{equation}
    \boldsymbol{k} = 
    (\{1,3\}^n), 
    (\{3,1\}^n), 
    (\{3,1\}^n,3), 
    (\{1,3\}^n,1),
\end{equation}
can be expressed as polynomials in Riemann zeta values. Here, $n$ is a nonnegative integer, and $\{a,b\}^n$ denotes the sequence obtained by repeating $a$ and $b$ alternately $n$ times, 
so that  
$(\{a,b\}^2) = (a, b, a, b)$.
As an example, the following result is established in \cite{HMS24}:  
\begin{theorem}[{\cite[Theorem 1.3 and Corollary 1.4]{HMS24}}]
    For a nonnegative integer $n$, we have
    \begin{multline}
        \zeta^*_{\mathcal{S}_3}(\{3,1\}^n)
        =
        \frac{2(-4)^n}{(4n+2)!}\pi^{4n}
        +
        (-1)^{n+1} 
        \sum_{\substack{
            n_0,n_1\geq 0
            \\
            n_0+n_1=2n
        }}
        \frac{(-1)^{n_0}2^{n_0-n_1+2}}{(2n_0+2)!    }
        \pi^{2n_0}\zeta^*(2n_1+1) t 
        \\
        +
        (-1)^n 
        \sum_{\substack{
            n_0,n_1,n_2\geq 0
            \\
            n_0+n_1+n_2=2n
        }}
        \frac{(-1)^{n_0}2^{n_0-n_1-n_2+2}}{(2n_0+2)!}\pi^{2n_0}
        \zeta^*(2n_1+1)\zeta^*(2n_2+1)t^2
    \end{multline}
    and by modulo $\pi^2$ reduction, we have
    \begin{multline}
        \label{eq: zetaS3(3,1) mod pi^2}
        \zeta^*_{\mathcal{S}_3}(\{3,1\}^n)
        \\
        \equiv 
        \delta_{n, 0}
        -
        2(-4)^{-n}\zeta^*(4n+1)t
        +
        2(-4)^{-n}
        \sum_{\substack{
            n_1,n_2\geq 0
            \\
            n_1+n_2=2n
        }}
        \zeta^*(2n_1+1)\zeta^*(2n_2+1)t^2
    \pmod{\pi^2}.
    \end{multline}
\end{theorem}
% In \cite{HMS24}, the authors show that 
% $
% \zeta^*_{\mathcal{S}_m}(\{3,1\}^n)
% $ 
% cannot (with our current understanding) be written as polynomials in Riemann zeta values when $m \geq 4$. 
% This is because, even if we reduce modulo $\pi^2$, the coefficient of $t^3$ in 
% $\zeta^*_{\mathcal{S}_4}(\{3,1\}^n)$
% contains MZVs of depth greater than $1$. 
% It should be noted, however, that the $\mathbb{Q}$-linear independence 
% of all MZVs has not yet been established. 
% Hence, while our present methods do not yield an expression 
% purely in terms of Riemann zeta values, 
% we cannot completely rule out the possibility 
% that further advances in understanding the linear relations 
% among MZVs might eventually lead to such an expression.
In \cite{HMS24}, the authors suggested that
it is reasonable to think that 
$
\zeta^*_{\mathcal{S}_m}(\{3,1\}^n)
$ 
cannot be written as polynomials in Riemann zeta values when $m \geq 4$.
This is because, even if we reduce modulo $\pi^2$, the coefficient of $t^3$ in 
$\zeta^*_{\mathcal{S}_4}(3,1,3,1)$
is equal to 
$
\frac{605}{4}\zeta(11)
+\frac{19}{4}\zeta(3)^2\zeta(5)
+2\zeta(3)\zeta(3,5)
-
2\zeta(3,3,5)
$.
Thus, we also do not examine 
$\zeta^*_{\mathcal{S}_m}(\{3,1\}^n)$
for $m \geq 4$ in this paper.

Moreover, \cite{HMS24} examines the case of the index $(\{1,3\}^n)$.
It is known that $\zeta^*_{\mathcal{S}_m}(\{1,3\}^n)$ can be expressed as polynomials in Riemann zeta values when $m \leq 2$:
for example, the following result is established in \cite[Theorem 1.1]{HMS24}.
\begin{theorem}{\cite[Theorem 1.1]{HMS24}}
    For a nonnegative integer $n$, we have
    \begin{multline}
        \label{eq: zetaS2(1,3)}
        \zeta^*_{\mathcal{S}_2}(\{1,3\}^n) 
        = 
        \frac{2(-4)^n}{(4n+2)!} \pi^{4n} 
        + 
        \biggl(
            \sum_{\substack{n_0,n_1 \geq 0 \\ n_0+n_1=n}} 
            \frac{(-4)^{n_0+1} (2 - (-4)^{-n_1})}{(4n_0+2)!} 
            \pi^{4n_0} 
            \zeta^*(4n_1+1) 
            \\
            - 
            (-1)^n 
            \sum_{\substack{
                n_0,n_1 \geq 0 
                \\
                n_0+n_1=2n 
                \\
                n_0, n_1 \text{:odd}
            }} 
            \frac{2^{n_0-n_1+2}}{(2n_0+2)!} 
            \pi^{2n_0} 
            \zeta(2n_1+1)
        \biggr) 
        t.
    \end{multline}
\end{theorem}
According to \cite{HMS24}, since 
the coefficient of $t^2$ in $\zeta^*_{\mathcal{S}_3}(1,3,1,3)$
is given by 
$
\frac{1}{2}\zeta(2)\zeta(3)\zeta(5)
+\zeta(2)\zeta(3,5)
-
\frac{1}{2}\zeta(3)^2\zeta(4)
-
\frac{1}{4}\zeta(3)\zeta(7)
+
\frac{81}{8}\zeta(5)^2
-
\frac{103}{10}\zeta(10)
$,
it is reasonable to believe that 
$\zeta^*_{\mathcal{S}_3}(1,3,1,3)$
cannot be written as polynomials in Riemann zeta values. 
Nonetheless,
it is conjectured that when reduced modulo $\pi^2$,
$\zeta^*_{\mathcal{S}_3}(\{1,3\}^n)$ can always 
be written as a polynomial in Riemann zeta values. 
This conjecture has not yet been proven.
Therefore, the goal of this paper is to give a detailed proof of the conjectured polynomial expression modulo $\pi^2$ for 
$\zeta^*_{\mathcal{S}_3}(\{1,3\}^n)$
in terms of Riemann zeta values, as stated in \cite[Conjecture 1.2]{HMS24}.
\begin{main theorem}[{\cite[Conjecture 1.2]{HMS24}}]
    \label{Thm. main}
    For a nonnegative integer $n$, we have
    \begin{equation}
        \begin{split}
            \label{eq:main}
            \zeta^*_{\mathcal{S}_3}(\{1,3\}^n)
            &\equiv 
            \delta_{n, 0}
            +
            2((-4)^{-n}-4)
            \zeta^*(4n+1)t 
            \\
            &\qquad
            +
            \biggl(
                -2(-4)^{-n}
                \sum_{\substack{
                    n_1,n_2\geq 0
                    \\
                    n_1+n_2=n-1
                }}
                \zeta(4n_1+3)
                \zeta(4n_2+3)
                \\
                &\qquad\qquad
                +2
                \sum_{\substack{
                    n_1,n_2\geq 0
                    \\
                    n_1+n_2=n
                }}
                ((-4)^{-n_1}-2)
                ((-4)^{-n_2}-2)
                \zeta^*(4n_1+1)
                \zeta^*(4n_2+1)
            \biggr)
            t^2 
    \pmod{\pi^2}.
        \end{split}
    \end{equation}
\end{main theorem}
We remark that since the coefficient of $t^3$ in
$\zeta^*_{\mathcal{S}_4}(1,3,1,3)$ 
is equal to 
$
-\frac{845}{4}\zeta(11)
-\frac{9}{4}\zeta(3)^2\zeta(5)
-\zeta(3)\zeta(3,5)
+2\zeta(3,3,5)
$
modulo $\pi^2$,
it makes sense to think that 
$\zeta^*_{\mathcal{S}_4}(1,3,1,3)$
cannot be written as polynomials in Riemann zeta values,
and hence we do not investigate 
$\zeta^*_{\mathcal{S}_m}(\{1,3\}^n)$ for $m \geq 4$ in this article.

We also believe that our main theorem, along with theorems in \cite{HMS24} for the $t$-adic SMZVs, 
should hold for the $\boldsymbol{p}$-adic finite MZVs as well. 
This idea comes from the generalized Kaneko--Zagier conjecture, 
which suggests that $t$-adic SMZVs and $\boldsymbol{p}$-adic finite MZVs satisfy relations
of the same form.
(for details, see \cite{KZ,OSY21}). 
Note that these relations for $\boldsymbol{p}$-adic finite MZVs have not been proved yet.

The structure of this paper is as follows.
In Section 2, we review the necessary facts and state them as lemmas. 
In Section 3, we prove the main theorem.
\section{Preparation for proof}
Let $\mathcal{I}$ be a $\mathbb{Q}$-vector space generated by all indices.
For all indices $\boldsymbol{k}, \boldsymbol{l}$ 
and for all positive integers $k,l$, a $\mathbb{Q}$-bilinear product $*$ on $\mathcal{I}$ is defined inductively by setting
\begin{align}
    \emptyset*\boldsymbol{k}&=\boldsymbol{k}*\emptyset=\boldsymbol{k},
    \\
    (\boldsymbol{k},k)*(\boldsymbol{l},l)
    &=
    ((\boldsymbol{k})*(\boldsymbol{l},l),k)
    +
    ((\boldsymbol{k},k)*(\boldsymbol{l}),l)
    +
    (\boldsymbol{k}*\boldsymbol{l},k+l).
\end{align}
Note that 
$\zeta(\boldsymbol{k}*\boldsymbol{l})=\zeta(\boldsymbol{k})\zeta(\boldsymbol{l})$
holds for any admissible indices $\boldsymbol{k},\boldsymbol{l}$.
For an index $\boldsymbol{k}=(k_1,\ldots,k_r)$ and nonnegative integers $m$ and $n$, 
we define a $\mathbb{Q}$-linear map ${}_mI_n:\mathcal{I}\rightarrow \mathcal{I}$ by 
\begin{equation}
    {}_mI_n(\boldsymbol{k})
    =
    {}_mI_n(k_1,\ldots,k_r)
    :=
    \sum_{i=0}^r
    (-1)^{k_r+\cdots+k_{i+1}}
    \sigma_m(k_1,\ldots,k_i)
    *
    \sigma_n(k_r,\ldots,k_{i+1}),
\end{equation} 
where
$\sigma_n:\mathcal{I}\rightarrow \mathcal{I}$ is a $\mathbb{Q}$-linear map defined by
\begin{equation}
    \sigma_n(\boldsymbol{k})
    =
    \sigma_n(k_1,\ldots,k_r)
    :=
    \sum_{\substack{
        l_1,\ldots,l_r\geq 0
        \\
        l_1+\cdots+l_r=n
    }}
    (
    k_1+l_1,\ldots,k_r+l_r
    )
    \prod_{i=1}^{r}
    \binom{k_i+l_i-1}{l_i}.
\end{equation}
For the empty index, we define $\sigma_n(\emptyset):=\delta_{n,0}$.
Hereafter, we often abbreviate ${}_0I_n$ as $I_n$.
Note that $\zeta^*(\sigma_n(\boldsymbol{k}))=\zeta^*_n(\boldsymbol{k})$ holds 
for every nonnegative integer and any index $\boldsymbol{k}$.
Using this notation and the definition of $\mathcal{S}_m$-SMZVs,
we see that $\zeta_{\mathcal{S}_3}^*(\boldsymbol{k})$ can be written as
\begin{equation}
    \label{eq: zetaS3=z(I0)+z(I1)t+z(I2)t^2}
    \zeta_{\mathcal{S}_3}^*(k_1,\ldots,k_r)
    =
    \zeta^*(I_0(k_1,\ldots,k_r))
    +\zeta^*(I_1(k_1,\ldots,k_r))t
    +\zeta^*(I_2(k_1,\ldots,k_r))t^2.
\end{equation}
\begin{lemma}[cf. {\cite[Lemma 2.23]{HMS24}}]
    \label{lem: sigma_n(k*l)=sum(sigma_i(k)*sigma_{n-i}(l))}
    Let $n$ be a nonnegative integer.
    For indices $\boldsymbol{k}$ and $\boldsymbol{l}$, we have
    \begin{equation}
        \sigma_n(\boldsymbol{k}*\boldsymbol{l})
        =
        \sum_{i=0}^n
        \sigma_i(\boldsymbol{k})
        *
        \sigma_{n-i}(\boldsymbol{l}).
    \end{equation}
\end{lemma}
\begin{proof}
    By induction on the sum of the lengths of the indices.
    If either $\boldsymbol{k}$ or $\boldsymbol{l}$ is empty, the statement is obviously true.
    Assume that
    \begin{align}
        \sigma_n(\boldsymbol{k}*(\boldsymbol{l},l))
        &=
        \sum_{i=0}^n
        \sigma_i(\boldsymbol{k})
        *
        \sigma_{n-i}(\boldsymbol{l},l),
        \label{eq: H.P. k*(l,l)}
        \\
        \sigma_n((\boldsymbol{k},k)*\boldsymbol{l})
        &=
        \sum_{i=0}^n
        \sigma_i(\boldsymbol{k},k)
        *
        \sigma_{n-i}(\boldsymbol{l}),
        \label{eq: H.P. (k,k)*(l)}
        \\
        \sigma_n(\boldsymbol{k}*\boldsymbol{l})
        &=
        \sum_{i=0}^n
        \sigma_i(\boldsymbol{k})
        *
        \sigma_{n-i}(\boldsymbol{l}),
        \label{eq: H.P. k*l}
    \end{align}
    all hold for indices $\boldsymbol{k}, \boldsymbol{l}$ and nonnegative integers $k,l$. 
    By the definition of $\sigma_n$, we see that 
    \begin{equation}
        \begin{split}
        \sigma_n(\boldsymbol{k},k)
        &=
        \sum_{\substack{
            a_1,\ldots,a_r,a \geq 0
            \\
            a_1+\cdots+a_r+a=n}
        }
        (k_1+a_1, \cdots ,k_r+a_r,k+a)
        \binom{k+a-1}{a}
        \prod_{i=1}^r
        \binom{k_i+a_i-1}{a_i}
        \\
        &=
        \sum_{a=0}^{n}
        \binom{k+a-1}{a}
        \bigl(
            \sigma_{n-a}(\boldsymbol{k}),k+a
        \bigr).
        \label{eq: another sigma_n(k,k)}
        \end{split}
    \end{equation}
    This representation will often be used in later calculations.
    Let us calculate the case of $n$.
    We have 
    \begin{equation}
        \begin{split}
            &\sum_{i=0}^n 
            \sigma_i(\boldsymbol{k},k)
            *
            \sigma_{n-i}(\boldsymbol{l},l)
            \\
            &=
            \sum_{i=0}^n
            \biggl(
                \sum_{a=0}^{i}
                \binom{k+a-1}{a}
                \bigl(\sigma_{i-a}(\boldsymbol{k}),k+a\bigr)
            \biggr)
            *
            \biggl(
                \sum_{b=0}^{n-i}
                \binom{l+b-1}{b}
                \bigl(\sigma_{n-i-b}(\boldsymbol{l}),l+b\bigr)
            \biggr)
            \\
            &=
            \sum_{i=0}^{n}
            \sum_{a=0}^{i}
            \sum_{b=0}^{n-i}
            \binom{k+a-1}{a}
            \binom{l+b-1}{b}
            \bigl(\sigma_{i-a}(\boldsymbol{k})*(\sigma_{n-i-b}(\boldsymbol{l}),l+b),k+a\bigr)
            \\
            &\qquad
            +
            \sum_{i=0}^{n}
            \sum_{a=0}^{i}
            \sum_{b=0}^{n-i}
            \binom{k+a-1}{a}
            \binom{l+b-1}{b}
            \bigl((\sigma_{i-a}(\boldsymbol{k}),k+a)*\sigma_{n-i-b}(\boldsymbol{l}),l+b\bigr)
            \\
            &\qquad
            +
            \sum_{i=0}^{n}
            \sum_{a=0}^{i}
            \sum_{b=0}^{n-i}
            \binom{k+a-1}{a}
            \binom{l+b-1}{b}
            \bigl(\sigma_{i-a}(\boldsymbol{k})*\sigma_{n-i-b}(\boldsymbol{l}),k+a+l+b\bigr).
        \end{split}
    \end{equation}
    The first term is given by 
    \begin{multline}
        \sum_{i=0}^{n}
        \sum_{a=0}^{i}
        \sum_{b=0}^{n-i}
        \binom{k+a-1}{a}
        \binom{l+b-1}{b}
        \biggl(
            \bigl(\sigma_{i-a}(\boldsymbol{k})*(\sigma_{n-i-b}(\boldsymbol{l}),l+b),k+a\bigr)
        \biggr)
        \\
        \begin{aligned}[c]
            &=
            \sum_{i=0}^{n}
            \sum_{a=0}^{i}
            \binom{k+a-1}{a}
            \bigl(
                \sigma_{i-a}(\boldsymbol{k})*\sigma_{n-i}(\boldsymbol{l},l),k+a
            \bigr)
            \\
            &=
            \sum_{a=0}^{n}
            \sum_{i=a}^{n}
            \binom{k+a-1}{a}
            \bigl(
                \sigma_{i-a}(\boldsymbol{k})*\sigma_{n-i}(\boldsymbol{l},l),k+a
            \bigr)
            \\
            &=
            \sum_{a=0}^{n}
            \binom{k+a-1}{a}
            \bigl(
                \sigma_{n-a}(\boldsymbol{k}*(\boldsymbol{l},l)),k+a
            \bigr)
            \\
            &=
            \sigma_n
            \bigl(
                \boldsymbol{k}*(\boldsymbol{l},l),k
            \bigr),
        \end{aligned}
    \end{multline}
    where we use \eqref{eq: another sigma_n(k,k)} in the first and last equality,
    and the induction hypothesis \eqref{eq: H.P. k*(l,l)} in the third equality. 
    Similarly, the second term is given by
    \begin{equation}
        \sum_{i=0}^{n}
        \sum_{a=0}^{i}
        \sum_{b=0}^{n-i}
        \binom{k+a-1}{a}
        \binom{l+b-1}{b}
        \bigl((\sigma_{i-a}(\boldsymbol{k}),k+a)*\sigma_{n-i-b}(\boldsymbol{l}),l+b\bigr)
        =
        \sigma_n 
        \bigl(
            (\boldsymbol{k},k)*\boldsymbol{l},l
        \bigr)
    \end{equation}
    by \eqref{eq: another sigma_n(k,k)} and the induction hypothesis \eqref{eq: H.P. (k,k)*(l)}.
    Setting $c=a+b$, we obtain the third term as follows:
    \begin{multline}
        \sum_{i=0}^{n}
        \sum_{a=0}^{i}
        \sum_{b=0}^{n-i}
        \binom{k+a-1}{a}
        \binom{l+b-1}{b}
        \bigl(
            \sigma_{i-a}(\boldsymbol{k})*\sigma_{n-i-b}(\boldsymbol{l}),k+a+l+b
        \bigr)
        \\
        \begin{aligned}[c]
            &=
            \sum_{c=0}^{n}
            \sum_{a=0}^{c}
            \sum_{i=a}^{n-c+a}
            \binom{k+a-1}{a}
            \binom{l+c-a-1}{c-a}
            \bigl(
                \sigma_{i-a}(\boldsymbol{k})*\sigma_{n-c+a-i}(\boldsymbol{l}),k+l+c
            \bigr)
            \\
            &=
            \sum_{c=0}^{n}
            \sum_{a=0}^{c}
            \binom{k+a-1}{a}
            \binom{l+c-a-1}{c-a}
            \bigl(
                \sigma_{n-c}(\boldsymbol{k}*\boldsymbol{l}),k+l+c
            \bigr)
            \\
            &=
            \sum_{c=0}^{n}
            \binom{k+l+c-1}{c}
            \bigl(
                \sigma_{n-c}(\boldsymbol{k}*\boldsymbol{l}),k+l+c
            \bigr)
            \\
            &=
            \sigma_n
            \bigl(
                \boldsymbol{k}*\boldsymbol{l},k+l
            \bigr),
        \end{aligned}
    \end{multline}
    where we use the induction hypothesis  \eqref{eq: H.P. k*l} in the second equality, 
    \eqref{eq: another sigma_n(k,k)} in the last equality, 
    and the following binomial identity in the third equality:
    \begin{equation}
        \sum_{a=0}^{c}
        \binom{k+a-1}{a}
        \binom{l+c-a-1}{c-a}
        =
        \binom{k+l+c-1}{c}.
    \end{equation}
    This identity follows from the Cauchy product of the generating functions of binomial coefficients:
    \begin{equation}
        \frac{1}{(1-x)^{n+1}} = \sum_{m\ge0} \binom{n+m}{m}\,x^m
        \in\mathbb{Z}\llbracket x \rrbracket 
        \qquad(n\in\mathbb{Z}_{\geq 0}).
    \end{equation}
    Hence, we finally obtain
    \begin{equation}
        \begin{split}
        \sum_{i=0}^n
        \sigma_i(\boldsymbol{k},k)
        *
        \sigma_{n-i}(\boldsymbol{l},l)
        &=
        \sigma_n
        \bigl(
            \boldsymbol{k}*(\boldsymbol{l},l),k
        \bigr)
        +
        \sigma_n 
        \bigl(
            (\boldsymbol{k},k)*\boldsymbol{l},l
        \bigr)
        +
        \sigma_n
        \bigl(
            \boldsymbol{k}*\boldsymbol{l},k+l
        \bigr)
        \\
        &=
        \sigma_n
        \bigl(
            (\boldsymbol{k},k)*(\boldsymbol{l},l)
        \bigr),
        \end{split}
    \end{equation}
    which proves the case of $n$, giving the desired result.
\end{proof}
We remark that the proof for the case of $n=1$ is given in \cite[Lemma 2.23]{HMS24}.
\begin{lemma}
    \label{lem: sum iIn-i=sn(I0(a,b))}
    For a nonnegative integer $n$ and an index $\boldsymbol{k}=(k_1,\ldots,k_r)$, we have 
    \begin{equation}
        \sum_{i=0}^n 
        {}_i I_{n-i}(\boldsymbol{k})
        =
        \sigma_n(I_0(\boldsymbol{k})).
    \end{equation}
\end{lemma}
\begin{proof}
    By Lemma \ref{lem: sigma_n(k*l)=sum(sigma_i(k)*sigma_{n-i}(l))}, we have
    \begin{equation}
        \begin{split}
            \sum_{i=0}^n 
            {}_i I_{n-i}(k_1,\ldots,k_r)
            &=
            \sum_{i=0}^n 
            \biggl(
                \sum_{j=0}^r 
                (-1)^{k_r+\cdots + k_{j+1}}
                \sigma_i(k_1,\ldots,k_{j})
                *
                \sigma_{n-i}(k_{r},\ldots,k_{j+1})
            \biggr)
            \\
            &=
            \sum_{j=0}^r 
            (-1)^{k_r+\cdots + k_{j+1}}
            \sum_{i=0}^n 
            \biggl(
                \sigma_i(k_1,\ldots,k_{j})
                *
                \sigma_{n-i}(k_{r},\ldots,k_{j+1})
            \biggr)
            \\
            &=
            \sum_{j=0}^r 
            (-1)^{k_r+\cdots + k_{j+1}}
            \sigma_n
            \bigl(
                (k_1,\ldots,k_{j})
                *
                (k_{r},\ldots,k_{j+1})
            \bigr)
            \\
            &=
            \sigma_n(I_0(k_1,\ldots,k_r)),
        \end{split}
    \end{equation}
    which implies the claim.
\end{proof}
\begin{lemma}
    \label{lem: I2+1I1+2I=s2(I0(a,b))}
    For an index $(k_1,\ldots,k_r)$ with $k_1+\cdots+k_r$ even, we have 
    \begin{equation}
        I_2(k_1,\ldots,k_r)
        +
        {}_1I_1(k_1,\ldots,k_r)
        +
        I_2(k_r,\ldots,k_1)
        =
        \sigma_2(I_0(k_1,\ldots,k_r)).
    \end{equation}
\end{lemma}
\begin{proof}
    This follows from Lemma \ref{lem: sum iIn-i=sn(I0(a,b))} and 
    \begin{equation}
        \begin{split}
            {}_2I_0(k_1,\ldots,k_r)
            &=
            (-1)^{k_1+\cdots+k_r}
            \sum_{i=0}^r 
            (-1)^{k_1+\cdots+k_{i}}
            (k_{r},\ldots,k_{i+1})
            *
            \sigma_2(k_1,\ldots,k_{i})
            \\
            &=
            {}_0I_2(k_r,\ldots,k_1).
        \end{split}
    \end{equation}
\end{proof}
Let $\mysh$ denote the shuffle of indices, e.g., $(a,b)\ \mysh\  (c)=(c,a,b)+(a,c,b)+(a,b,c) $.
\begin{lemma}[cf.\ {\cite[Lemma 2.16]{HMS24}}]
    \label{lem: b mysh a and bb mysh a}
    Let $n$ be a nonnegative integer.
    If $a$ and $b$ are positive integers, then we have
    \begin{align}
        \label{eq:b1a}
        (b)\widetilde{\sh}(\{a\}^n)
        &=
        \sum_{i=0}^n
        (-1)^i
        (ai+b)*(\{a\}^{n-i}),
        \\
        \label{eq:b2a}
        (b,b)\widetilde{\sh}(\{a\}^n)
        &=
        \sum_{\substack{
            i,j\geq 0
            \\
            0\leq i+j\leq n
        }}
        (-1)^{i+j}
        (ai+b,aj+b)*(\{a\}^{n-i-j}).
    \end{align}
\end{lemma}
\begin{proof}
    The proof of \eqref{eq:b1a} can be found in \cite[Lemma 2.16]{HMS24}, so we omit it here.
    We now proceed to prove \eqref{eq:b2a}.
    The right-hand side of \eqref{eq:b2a} can be written as
    \begin{multline}
        \label{eq: (ai+b,aj+b)*(a^n)}
        \sum_{\substack{
            i,j\geq 0
            \\
            0\leq i+j\leq n
        }}
        (-1)^{i+j}
        (ai+b,aj+b)*(\{a\}^{n-i-j})
        \\
        \begin{aligned}[c] 
            &=
            \sum_{\substack{
                i,j\geq 0
                \\
                0\leq i+j\leq n-2
            }}        
            (-1)^{i+j}
            (ai+b,aj+b)\widetilde{\sh}(\{a\}^{n-i-j})
            \\&\qquad
            +
            \sum_{\substack{
                i,j\geq 0
                \\
                0\leq i+j\leq n-2
            }}        
            (-1)^{i+j}
            (a(i+1)+b,aj+b)\widetilde{\sh}(\{a\}^{n-i-1-j})
            \\&\qquad
            +
            \sum_{\substack{
                i,j\geq 0
                \\
                0\leq i+j\leq n-2
            }}        
            (-1)^{i+j}
            (ai+b,a(j+1)+b)\widetilde{\sh}(\{a\}^{n-i-j-1})
            \\&\qquad
            +
            \sum_{\substack{
                i,j\geq 0
                \\
                0\leq i+j\leq n-2
            }}        
            (-1)^{i+j}
            (a(i+1)+b,a(j+1)+b)\widetilde{\sh}(\{a\}^{n-i-j-2})
            \\&\qquad
            +
            (-1)^{n-1}
            \sum_{\substack{ i,j\geq 0\\i+j=n-1}}
            (ai+b,aj+b)*(a)
            +
            (-1)^n
            \sum_{\substack{ i,j\geq 0\\i+j=n}}
            (ai+b,aj+b).
        \end{aligned}
    \end{multline}
    The first term in \eqref{eq: (ai+b,aj+b)*(a^n)} is given by 
    \begin{multline}
        \label{eq:1 in (ai+b,aj+b)*(a^n)}
        \sum_{\substack{i,j\geq 0 \\ 0\leq i+j\leq n-2}}
        (-1)^{i+j}
        (ai+b,aj+b)\widetilde{\sh}(\{a\}^{n-i-j})
        \\
        =
        \sum_{\substack{i,j\geq 0 \\ 0\leq i+j\leq n-1}}
        (-1)^{i+j}
        (ai+b,aj+b)\widetilde{\sh}(\{a\}^{n-i-j})
        -(-1)^{n-1}
        \sum_{\substack{i,j\geq 0 \\i+j=n-1}}
        (ai+b,aj+b)\widetilde{\sh}(a).
    \end{multline}
    By shifting as $i+1\mapsto i$, the second term in \eqref{eq: (ai+b,aj+b)*(a^n)} is simplified to
    \begin{multline}
        \label{eq:2 in (ai+b,aj+b)*(a^n)}
        \sum_{\substack{i,j\geq 0 \\ 0\leq i+j\leq n-2}}
        (-1)^{i+j}
        (a(i+1)+b,aj+b)\widetilde{\sh}(\{a\}^{n-i-1-j})
        \\
        \begin{aligned}[c]
            &=
            -
            \sum_{\substack{
                i\geq 1
                ,
                j\geq 0
                \\
                1\leq i+j\leq n-1
            }}
            (-1)^{i+j}
            (ai+b,aj+b)\widetilde{\sh}(\{a\}^{n-i-j})
            \\
            &=
            -
            \sum_{\substack{i,j\geq 0 \\ 0\leq i+j\leq n-1}}
            (-1)^{i+j}
            (ai+b,aj+b)\widetilde{\sh}(\{a\}^{n-i-j})
            +
            \sum_{j=0}^{n-1}
            (-1)^{j}
            (b,aj+b)\widetilde{\sh}(\{a\}^{n-j}).
        \end{aligned}
    \end{multline}
    Similarly, by putting $j+1\mapsto j$ the third term in \eqref{eq: (ai+b,aj+b)*(a^n)} is given by
    \begin{multline}
        \label{eq:3 in (ai+b,aj+b)*(a^n)}
        \sum_{\substack{i,j\geq 0 \\ 0\leq i+j\leq n-2}}
        (-1)^{i+j}
        (ai+b,a(j+1)+b)\widetilde{\sh}(\{a\}^{n-i-j-1})
        \\
        =
        -
        \sum_{\substack{i,j\geq 0 \\ 0\leq i+j\leq n-1}}
        (-1)^{i+j}
        (ai+b,aj+b)\widetilde{\sh}(\{a\}^{n-i-j})
        +
        \sum_{i=0}^{n-1}
        (-1)^{i}
        (ai+b,b)\widetilde{\sh}(\{a\}^{n-i}).
    \end{multline}
    By applying the transformations $i+1\mapsto i$  and  $j+1 \mapsto j$,
    the fourth term in \eqref{eq: (ai+b,aj+b)*(a^n)} is given by
    \begin{multline}
        \label{eq:4 in (ai+b,aj+b)*(a^n)}
        \sum_{\substack{i,j\geq 0 \\ 0\leq i+j\leq n-2}}
            (-1)^{i+j}
            (a(i+1)+b,a(j+1)+b)\widetilde{\sh}(\{a\}^{n-i-j-2})
            \\
            \begin{aligned}[c]
                &=
                \sum_{\substack{i,j\geq 0 \\ 0\leq i+j\leq n-1}}
                (-1)^{i+j}
                (ai+b,aj+b)\widetilde{\sh}(\{a\}^{n-i-j})
                +
                \sum_{\substack{
                    i,j\geq 0
                    \\
                    i+j=n
                }}
                (-1)^{n}
                (ai+b,aj+b)
                \\
                &\qquad
                -
                (b,b)\mysh(\{a\}^n)
                -
                \sum_{i=1}^n
                (-1)^i
                (ai+b,b)\mysh(\{a\}^{n-i})
                -
                \sum_{j=1}^n
                (-1)^j
                (b,aj+b)\mysh(\{a\}^{n-j}).
            \end{aligned}
    \end{multline}
    Thus, by applying 
    \eqref{eq:1 in (ai+b,aj+b)*(a^n)}, 
    \eqref{eq:2 in (ai+b,aj+b)*(a^n)}, 
    \eqref{eq:3 in (ai+b,aj+b)*(a^n)} and 
    \eqref{eq:4 in (ai+b,aj+b)*(a^n)},
    we find that 
    \eqref{eq: (ai+b,aj+b)*(a^n)} is equal to
    \begin{multline}
        (-1)^{n-1}
        \sum_{\substack{i,j\geq 0 \\i+j=n-1}}
        (ai+b,aj+b)*(a)
        +
        (-1)^n
        \sum_{\substack{i,j\geq 0\\i+j=n}}
        (ai+b,aj+b)
        \\
        \begin{aligned}[c]
            &\quad
            -(-1)^{n-1}
            \sum_{\substack{i,j\geq 0 \\i+j=n-1}}
            (ai+b,aj+b)\widetilde{\sh}(a)
            \\
            &\quad
            +
            \sum_{j=0}^{n-1}
            (-1)^{j}
            (b,aj+b)\widetilde{\sh}(\{a\}^{n-j})
            +
            \sum_{i=0}^{n-1}
            (-1)^{i}
            (ai+b,b)\widetilde{\sh}(\{a\}^{n-i})
            \\
            &\quad
            +
            \sum_{\substack{i,j\geq 0\\i+j=n}}
            (-1)^{n}
            (ai+b,aj+b)
            \\
            &\qquad
            -
            (b,b)\mysh(\{a\}^n)
            -
            \sum_{i=1}^n
            (-1)^i
            (ai+b,b)\mysh(\{a\}^{n-i})
            -
            \sum_{j=1}^n
            (-1)^j
            (b,aj+b)\mysh(\{a\}^{n-j})
            \\
            &=
            (-1)^{n-1}
            \sum_{\substack{i,j\geq 0 \\i+j=n-1}}
            \biggl(
                (a(i+1)+b,aj+b)
                +
                (ai+b,a(j+1)+b)
            \biggr)
            \\
            &\qquad
            +2
            \sum_{\substack{i,j\geq 0\\i+j=n}}
            (-1)^n
            (ai+b,aj+b)
            +(b,b)\mysh(\{a\}^n)
            +(-1)^n(an+b,b)
            +(-1)^n(b,an+b).
        \end{aligned}
    \end{multline}
    Since
    \begin{equation}
        \begin{split}
            \sum_{\substack{i,j\geq 0 \\i+j=n-1}}
            (a(i+1)+b,aj+b)
            &=
            \sum_{\substack{
                i,j\geq 0
                \\
                i+j=n
            }}
            (ai+b,aj+b)
            -(b,an+b),
        \end{split}
    \end{equation}
    we have
    \begin{multline}
        \sum_{
            \substack{
                i,j\geq 0
                \\
                0\leq i+j\leq n
            }}
        (-1)^{i+j}
        (ai+b,aj+b)*(\{a\}^{n-i-j})
        \\
        \begin{aligned}[c]
            &=
            2\sum_{\substack{i,j\geq 0\\i+j=n}}
            (-1)^{n-1}
            (ai+b,aj+b)
            -
            (-1)^{n-1}(an+b,b)
            -
            (-1)^{n-1}(b,an+b)
            \\
            &\qquad
            +2
            \sum_{\substack{i,j\geq 0\\i+j=n}}
            (-1)^n
            (ai+b,aj+b)
            +(b,b)\mysh(\{a\}^n)
            +(-1)^n(an+b,b)
            +(-1)^n(b,an+b)
            \\
            &=
            (b,b)\mysh(\{a\}^n).
        \end{aligned}
    \end{multline}
    Therefore, we obtain the required formula.
\end{proof}
\begin{lemma}
    \label{lem:sigma2 k^n}
    For positive integers $k$ and $n$, we have
    \begin{equation}
        \begin{split}
            \sigma_2(\{k\}^n)
            &=
            \frac{(k+1)k}{2}
            \sum_{i=0}^{n-1}
            (-1)^i
            (k(i+1)+2)*(\{k\}^{n-i-1})
            \\
            &\qquad
            +k^2
            \sum_{\substack{i,j\geq 0 \\ 0\leq i+j\leq n-2}}
            (-1)^{i+j}
            (k(i+1)+1,k(j+1)+1)
            *
            (\{k\}^{n-i-j-2}).
        \end{split}
    \end{equation}
\end{lemma}
\begin{proof}
    The definition of $\sigma_2$ and Lemma \ref{lem: b mysh a and bb mysh a} together lead to the assertion.
\end{proof}
\begin{lemma}[cf.\ {\cite[Lemma 2.24]{HMS24}}]
    \label{lem: I0(a,b)=(-1)^n(a+b)^n}
    Let $n$ be a nonnegative integer.
    If $a$ and $b$ are odd positive integers, then we have 
    \begin{equation}
        I_0(\{a,b\}^n)
        = 
        (-1)^n(\{a+b\}^n).
    \end{equation}
\end{lemma}
\begin{proof}
    The proof can be found in \cite[Lemma 2.24]{HMS24}.
    Thus, we omit it here.
\end{proof}
Next, we rephrase Lemma \ref{lem: I2+1I1+2I=s2(I0(a,b))} in terms of MZVs.
\begin{lemma}
    \label{lem: zeta(I2+1I1+2I)}
    Let $n$ be a positive integer.
    If $a$ and $b$ are odd positive integers, then we have 
    \begin{multline}
        \label{eq: I2+1I1+2I}
        \zeta^*
        (I_2(\{a,b\}^n))
        +
        \zeta^*
        ({}_1I_1(\{a,b\}^n))
        +
        \zeta^*
        (I_2(\{b,a\}^n))
        \\\equiv
        \frac{(a+b)^2}{2}
        \sum_{\substack{
            n_1,n_2\geq 0
            \\
            n_1+n_2=n
        }}
        \zeta^*\bigl((a+b)n_1+1\bigr)
        \zeta^*\bigl((a+b)n_2+1\bigr)
        \pmod{\pi^2}.
    \end{multline}
\end{lemma}
\begin{proof}
    By Lemma \ref{lem: I2+1I1+2I=s2(I0(a,b))}, we see that
    \begin{equation}
        \label{eq: I2(1,3)+1I1(1,3)+2I(1,3)=s2(I0(1,3))}
        I_2(\{a,b\}^n)
        +
        {}_1I_1(\{a,b\}^n)
        +
        I_2(\{b,a\}^n)
        =
        \sigma_2
        (I_0(\{a,b\}^n)).
    \end{equation}
    By Lemma \ref{lem:sigma2 k^n} and \ref{lem: I0(a,b)=(-1)^n(a+b)^n}, we have
    \begin{multline}
        \sigma_2
        (I_0(\{a,b\}^n))
        \\
        \begin{aligned}[c]
            &=
            (-1)^n
            \sigma_2
            (\{a+b\}^n)
            \\
            &=
            (-1)^n
            \frac{(a+b+1)(a+b)}{2}
            \sum_{i=0}^{n-1}
            (-1)^i
            \bigl((a+b)(i+1)+2\bigr)
            *
            \bigl(\{a+b\}^{n-i-1}\bigr)
            \\
            &\qquad
            +(-1)^n
            (a+b)^2
            \sum_{\substack{i,j\geq 0 \\ 0\leq i+j\leq n-2}}
            (-1)^{i+j}
            \bigl((a+b)(i+1)+1,(a+b)(j+1)+1\bigr)*\bigl(\{a+b\}^{n-i-j-2}\bigr).
            \label{eq: s2(I0(1,3))}
        \end{aligned}
    \end{multline}
    The equation above, 
    together with the fact that 
    $\zeta^*(\{a+b\}^n)$ is congruent to zero modulo $\pi^2$
    for every positive integer $n$,
    implies that
    \begin{equation}
        \begin{split}
            \zeta^*
            (
                \sigma_2
                (I_0(\{a,b\}^n))
            )
            &\equiv
            (a+b)^2
            \sum_{\substack{
                n_1,n_2\geq 0
                \\
                n_1+n_2=n-2
            }}
            \zeta^*
            \bigl(
                (a+b)(n_1+1)+1,(a+b)(n_2+1)+1
            \bigr)
            \pmod{\pi^2}
            \\
            % \\
            % &\qquad
            % +(-1)^n
            % 16
            % \sum_{\substack{
            %     n_0,n_1,n_2\geq 0
            %     \\
            %     n_0+n_1+n_2=n-2
            % }}
            % (-1)^{n_1+n_2}
            % \frac{2^{2n_0+1}\pi^{4n_0}}{(4n_0+2)!}
            % \zeta(4n_1+5,4n_2+5).
        \end{split}
    \end{equation}
    Since 
    \begin{multline}
        \zeta^*
        \bigl(
            (a+b)(n_1+1)+1
        \bigr)
        \zeta^*
        \bigl(
            (a+b)(n_2+1)+1
        \bigr)
        \\
        \equiv
        \zeta^*
        \bigl(
            (a+b)(n_1+1)+1
            ,
            (a+b)(n_2+1)+1
        \bigr)
        +
        \zeta^*
        \bigl(
            (a+b)(n_2+1)+1
            ,
            (a+b)(n_1+1)+1
        \bigr)
        \pmod{\pi^2},
    \end{multline}
    we have
    \begin{equation}
        \begin{split}
            \label{eq:zeta^*(sigma_2(I_0(1,3)))}
            \zeta^*
            (
                \sigma_2
                (I_0(\{a,b\}^n))
            )
            &\equiv
            \frac{(a+b)^2}{2}
            \sum_{\substack{
                n_1,n_2\geq 0
                \\
                n_1+n_2=n-2
            }}
            \zeta^*
            \bigl(
                (a+b)(n_1+1)+1
            \bigr)
            \zeta^*
            \bigl(
                (a+b)(n_2+1)+1
            \bigr)
            \pmod{\pi^2}
            \\
            &=
            \frac{(a+b)^2}{2}
            \sum_{\substack{
                n_1,n_2\geq 0
                \\
                n_1+n_2=n
            }}
            \zeta^*
            \bigl(
                (a+b)n_1+1
            \bigr)
            \zeta^*
            \bigl(
                (a+b)n_2+1
            \bigr),
        \end{split}
    \end{equation}
    where we use $\zeta^*(1)=0$ in the last equality.
    Applying $\zeta^*$ on both sides of \eqref{eq: I2(1,3)+1I1(1,3)+2I(1,3)=s2(I0(1,3))},
    we obtain the claim of the lemma, as the right-side is given by 
    \eqref{eq:zeta^*(sigma_2(I_0(1,3)))}.
\end{proof}
The following formulas will be useful for later proofs.
For nonnegative integers $n$, we have
\begin{align}
    \zeta(\{4\}^n)
    &=
    \frac{2^{2n+1}\pi^{4n}}{(4n+2)!}
    \equiv 
    \delta_{n,0}
    \pmod{\pi^2},
    \label{eq: zeta(4^n)}
    \\
    \zeta(\OT n)
    &=
    \frac{1}{4^n}\zeta(\{4\}^n)
    \equiv 
    \delta_{n,0}
    \pmod{\pi^2},
    \label{eq: zeta(OT n)}
    \\
    \zeta^*(\{1,3\}^{n},1)
    &=
    \frac{2}{4^{n}}
    \sum_{j=1}^{n} (-1)^j \zeta(4j + 1) \zeta(\{4\}^{n-j})
    \equiv
    2(-4)^{-n} \zeta(4n+1)\pmod{\pi^2}.
    \label{eq: zeta((1,3)^n,1)}
\end{align}
See \cite[Example 2.2]{BBBL01} for the first and second formulas, 
and \cite[Proposition 4.2]{BY18} for the proof of the third.
\begin{lemma}
    \label{lem: zeta1(3,1)=zeta1(313)}
    For a nonnegative integer $n$, we have 
    \begin{equation}
        \label{eq: zeta1(3,1)=zeta1(313)}
        \zeta_1^*(\TO{n+1})
        \equiv
        \sum_{i=0}^n
        2(-4)^{-i}
        \zeta^*(4i+1)
        \zeta_1^*(\{3,1\}^{n-i},3)
        +
        2((-4)^{-n-1}-2)\zeta(4n+5)
        \pmod{\pi^2}.
    \end{equation}
\end{lemma}
\begin{proof}
    By definition of $I_1$, \eqref{eq: zeta(OT n)} and \eqref{eq: zeta((1,3)^n,1)},
    we have the following expression of $\zeta^*(I_1(\OT {n+1}))$
    for nonnegative integer $n$.
    \begin{align}
            \zeta^*(I_1(\OT {n+1}))
            &=
            \zeta_1^*(\TO{n+1})
            +\sum_{i=1}^{n+1} 
            \biggl(
                \zeta(\OT i)\zeta_1^*(\TO{n+1-i})
                -
                \zeta^*(\{1,3\}^{i-1},1)\zeta^*_1(\{3,1\}^{n+1-i},3)
            \biggr)
            \\
            &\equiv
            \zeta_1^*(\TO {n+1})
            -
            \sum_{i=1}^{n+1}
            2(-4)^{-i+1} \zeta^*(4(i-1)+1)\zeta_1^*(\{3,1\}^{n+1-i},3)
            \pmod{\pi^2}
            \\
            &=
            \zeta_1^*(\TO {n+1})
            -
            \sum_{i=0}^{n}
            2(-4)^{-i} \zeta^*(4i+1)\zeta_1^*(\{3,1\}^{n-i},3).\label{eq:2}
    \end{align}
    Moreover, By \eqref{eq: zetaS2(1,3)} and \eqref{eq: zetaS3=z(I0)+z(I1)t+z(I2)t^2}, we see that
    \begin{multline}
        \label{eq:1}
        \zeta^*(I_1(\OT {n+1}))
        \\
        \begin{aligned}[c]
            &=
            \sum_{\substack{
                n_0, n_1 \geq 0 
                \\
                n_0 + n_1 = n+1
            }} 
            \frac{(-4)^{n_0+1} (2 - (-4)^{-n_1})}{(4n_0 + 2)!} 
            \pi^{4n_0} \zeta^*(4n_1 + 1) 
            % \\
            % &\qquad 
            - (-1)^{n+1} 
            \sum_{\substack{n_0, n_1 \geq 0 \\ n_0 + n_1 = 2n+2 \\ n_0, n_1: \text{odd}}} 
            \frac{2^{n_0 - n_1 + 2}}{(2n_0 + 2)!} 
            \pi^{2n_0} \zeta(2n_1 + 1)
            \\
            &\equiv
            2((-4)^{-n-1}-2)\zeta(4n+5)
            \pmod{\pi^2}.
        \end{aligned}
    \end{multline}
    By comparing \eqref{eq:2} and \eqref{eq:1}, we obtain the assertion.
\end{proof}
\begin{proposition}
    \label{prop: zeta(1I1(1,3))}
    For a nonnegative integer $n$, we have 
    \begin{equation}
        \label{eq: zeta(1I1(1,3))}
        \zeta^*({}_1I_1(\{1,3\}^n))
        \equiv 
        -4
        \sum_{\substack{
            n_1,n_2\geq 0
            \\
            n_1+n_2=n
        }}
        \Bigl(
            (-4)^{-n}
            -
            (-4)^{-n_1}
            -
            (-4)^{-n_2}
        \Bigr)
        \zeta^*(4n_1+1)
        \zeta^*(4n_2+1)
        \pmod{\pi^2}.
    \end{equation}
\end{proposition}
\begin{proof}
    To prove \eqref{eq: zeta(1I1(1,3))}, we show that the generating functions on both sides coincide.
    Let
    \begin{align}
        P&=
        \sum_{n\geq 0}
        (-4)^{-n}\zeta^*(4n+1)u^{4n+1}
        \in \mathbb{R}\llbracket u\rrbracket,
        \\
        Q&=
        \sum_{n\geq 0}
        \zeta^*(4n+1)u^{4n+1}
        \in \mathbb{R}\llbracket u\rrbracket.
    \end{align}
    From these generating functions, 
    it is immediately clear that the right-hand side of \eqref{eq: zeta(1I1(1,3))} can be written as follows:
    \begin{equation}
        \label{eq: RHS zeta(1I1(1,3)) generating function}
        \sum_{n\geq 0}
        \biggl(
        -4
        \sum_{\substack{
            n_1,n_2\geq 0
            \\
            n_1+n_2=n
        }}
        \Bigl(
            (-4)^{-n}
            -
            (-4)^{-n_1}
            -
            (-4)^{-n_2}
        \Bigr)
        \zeta^*(4n_1+1)
        \zeta^*(4n_2+1)
        \biggr)
        u^{4n+2}
        = -4P(P-2Q).
    \end{equation}
    Next, we consider the left-hand side of \eqref{eq: zeta(1I1(1,3))}.
    By the definition of ${}_1I_1$, 
    it follows that the generating function of 
    the right-hand side of \eqref{eq: zeta(1I1(1,3))}
    can be expressed as the difference of products of the generating functions.
    \begin{multline}
        \label{eq: zeta(1I1(1,3)) generating function}
        \sum_{n\geq 0}
        \zeta^*({}_1I_1(\OT n))
        u^{4n+2}
        =
        \Biggl(
            \sum_{n_1\geq 0}
            \zeta_1^*(\TO{n_1})
            u^{4n_1+1}
        \Biggr)
        \Biggl(
            \sum_{n_2\geq 0}
            \zeta_1^*(\OT{n_2})
            u^{4n_2+1}
        \Biggr)
        \\-
        \biggl(
            \sum_{n_3\geq 0}
            \zeta_1^*(\OT {n_3}, 1)
            u^{4n_3+2}
        \biggr)
        \biggl(
            \sum_{n_4\geq 0}
            \zeta_1^*(\TO {n_4}, 3)
            u^{4n_4+4}
        \biggr).
    \end{multline}
    % Here, it is known \cite[proof of Lemma 2.15]{HMS24} that $ \zeta^*_1(\{1,3\}^{n})=-\zeta^*(\{1,3\}^n,1)$, 
    % and is given \cite[Proposition 4.2]{BY18} by
    According to \cite[proof of Lemma 2.15]{HMS24}, 
    it is known that $\zeta^*_1(\{1,3\}^{n})=-\zeta^*(\{1,3\}^n,1)$, 
    and from \eqref{eq: zeta((1,3)^n,1)}, we obtain
    \begin{equation}
        \label{eq: zeta_1(1,3) mod pi^2}
        \zeta_1^*(\{1,3\}^{n})
        \equiv 
        -2 
        (-4)^{-n}
        \zeta^*(4n+1)
        \pmod{\pi^2}
    \end{equation}
    for any nonnegative integer $n$.
    Hence, we have 
    \begin{equation}
        \label{eq: GF zeta_1(1,3) mod pi^2}
        \sum_{n_2\geq 0}
        \zeta_1^*(\OT{n_2})
        u^{4n_2+1}
        \equiv
        -2P
        \pmod{\pi^2}.
    \end{equation}
    Furthermore, in \cite[Lemma 2.19]{HMS24}, it is shown that
    \begin{multline}
        \zeta_1^*(\{1,3\}^{n}, 1)
            \\
            =
            (-4)^{-n}
            \sum_{i=0}^{n}
            (-1)^{n-i}
            (4i+1)
            \zeta(4i+2)
            \zeta(\{4\}^{n-i})
            +(-4)^{1-n}
            \sum_{\substack{
                n_1,n_2,n_3\geq 0
                \\
                n_1+n_2+n_3=n
            }}
            (-1)^{n_3}
            \zeta^*(4n_1+1)
            \zeta^*(4n_2+1)
            \zeta(\{4\}^{n_3}).
    \end{multline}
    By reducing modulo $\pi^2$, it takes the following form.
    \begin{equation}
        \label{eq: zeta_1(1,3,1) mod pi^2}
        \zeta_1^*(\OT n ,1)
        \equiv
        -4
        \sum_{\substack{
            n_1,n_2\geq 0
            \\
            n_1+n_2=n
        }}
        (-4)^{-n}
        \zeta^*(4n_1+1)
        \zeta^*(4n_2+1)
        \pmod{\pi^2}.
    \end{equation}
    Thus, we have 
    \begin{equation}
        \label{eq: GF zeta_1(1,3,1) mod pi^2}
        \sum_{n_3\geq 0}
        \zeta_1^*(\OT {n_3}, 1)
        u^{4n_3+2}
        \equiv
        -4P^2 
        \pmod{\pi^2}.
    \end{equation}
    By \eqref{eq: GF zeta_1(1,3) mod pi^2} and \eqref{eq: GF zeta_1(1,3,1) mod pi^2},
    we see that \eqref{eq: zeta(1I1(1,3)) generating function} can be expressed as in the following. 
    \begin{multline}
        \sum_{n\geq 0}
        \zeta^*({}_1I_1(\OT n))
        u^{4n+2}
        \\
        \equiv 
        -2P
        \sum_{n\geq 0}
        \Biggl(
            \zeta_1^*(\TO{n+1})
        -
            \sum_{i=0}^{n}
            2(-4)^{-i}
            \zeta^*(4i+1)
            \zeta_1^*(\TO{n-i}, 3)
        \biggr)
        u^{4n+5}
        \pmod{\pi^2}.
    \end{multline}
    By Lemma \ref{lem: zeta1(3,1)=zeta1(313)}, we obtain
    \begin{equation}
        \begin{split}
            \label{eq: LHS zeta(1I1(1,3)) generating function}
            \sum_{n\geq 0}
            \zeta^*({}_1I_1(\OT n))
            u^{4n+2}
            &\equiv
            -4P
            \biggl(
                \sum_{n\geq 0} ((-4)^{-n-1}-2)\zeta(4n+5) u^{4n+5}
            \biggr)
            \pmod{\pi^2}
            \\
            &=
            -4P(P-2Q).
        \end{split}
    \end{equation}
    Thus \eqref{eq: RHS zeta(1I1(1,3)) generating function} and \eqref{eq: LHS zeta(1I1(1,3)) generating function} 
    together give the assertion.
\end{proof}
Note that \eqref{eq: zeta(1I1(1,3))} gives the coefficient of the $-st$ term in the 
expansion of $(s,t)$-adic SMZVs with $\boldsymbol{k}=(\{1,3\}^n)$.
Here, $(s,t)$-adic SMZVs are defined by 
\begin{equation}
    \zeta_{\HS}^{s,t}(\boldsymbol{k})
    =
    \zeta_{\HS}^{s,t}(k_1,\ldots,k_r)
    :=
    \sum_{i=0}^r 
    (-1)^{k_r+\cdots+k_{i+1}}
    \sum_{m,n=0}^{\infty}
    \zeta^*_m(k_1,\ldots,k_i)
    \zeta^*_n(k_{r},\ldots,k_{i+1})
    (-s)^m
    t^n
    \bmod{\pi^2},
\end{equation} 
which was introduced in \cite{HK23}.
Moreover, it also gives the coefficient of the $t_1 t_2$ term (modulo $\pi^2$)
in unified multiple zeta functions with $\boldsymbol{k}=(\{1,3\}^n)$.
Unified multiple zeta functions are explicitly given by the following formula.
\begin{equation}
    \zeta_{\HU}(\boldsymbol{s};t_1;t_2)
    =
    \zeta_{\HU}(s_1,\ldots,s_r;t_1;t_2)
    :=
    \sum_{i=0}^r 
    (-1)^{s_r+\cdots+s_{i+1}}
    \zeta(s_1,\ldots,s_i;t_1)
    \zeta(s_{r},\ldots,s_{i+1};t_2).
\end{equation}
For the detailed definition and the meaning of the notation in the equation above, see \cite{Komori21}.
\section{Proof of the main theorem}
Now we are ready to prove the main theorem.
\begin{proof}[Proof of the main theorem]
    Put $\boldsymbol{k}=(\{1,3\}^n)$ for a nonnegative integer $n$.
    Then \eqref{eq: zetaS3=z(I0)+z(I1)t+z(I2)t^2} is rewritten as
    \begin{equation}
        \label{eq:k=(1,3)ver  zetaS3=z(I0)+z(I1)t+z(I2)t^2}
        \zeta_{\mathcal{S}_3}^*(\{1,3\}^n)
        =
        \zeta^*(I_0(\{1,3\}^n))
        +\zeta^*(I_1(\{1,3\}^n))t
        +\zeta^*(I_2(\{1,3\}^n))t^2.
    \end{equation}
    By \eqref{eq: zetaS2(1,3)}, 
    we see that
    \begin{align}
        \zeta^*(I_0(\{1,3\}^n))
        &=
        \frac{2(-4)^n}{(4n+2)!} \pi^{4n}
        % \\
        % &
        \equiv 
        \delta_{n,0}
        \pmod{\pi^2},
        \label{eq: zeta^*(I_0(1,3))}
        \\
        \zeta^*(I_1(\{1,3\}^n))
        &=
        \sum_{\substack{n_0, n_1 \geq 0 \\ n_0 + n_1 = n}}
        \frac{(-4)^{n_0+1} (2 - (-4)^{-n_1})}{(4n_0+2)!} 
        \pi^{4n_0} \zeta^*(4n_1 + 1)
        \\
        &\qquad
        - (-1)^n \sum_{\substack{n_0, n_1 \geq 0 \\ n_0 + n_1 = 2n \\ n_0, n_1: \text{odd}}}
        \frac{2^{n_0-n_1+2}}{(2n_0+2)!} 
        \pi^{2n_0} \zeta(2n_1+1)
        \\
        &\equiv
        2((-4)^{-n}-4)\zeta^*(4n+1)
        \pmod{\pi^2}
        \label{eq: zeta^*(I_1(1,3))}
    \end{align}
    for every nonnegative integer $n$.
    Therefore, to prove the main theorem \eqref{eq:main}, it is sufficient to calculate $\zeta^*(I_2(\{1,3\}^n))$.
    Since it is obviously $\zeta^*(I_2(\{1,3\}^0))=0$, we assume that $n\geq 1$.
    The coefficient of $t^2$ in \eqref{eq: zetaS3(3,1) mod pi^2} is given by
    \begin{equation}
        \begin{split}
            \label{eq: zeta(I_2(3,1))}
            \zeta^*(I_2(\{3,1\}^n))
            &\equiv
            2(-4)^{-n} 
            \sum_{\substack{n_1, n_2 \geq 0 \\ n_1 + n_2 = 2n}} 
            \zeta^*(2n_1 + 1) 
            \zeta^*(2n_2 + 1)
            \pmod{\pi^2}
            \\
            &=
            2(-4)^{-n}
            \sum_{\substack{
             n_1,n_2\geq 0
             \\
             n_1+n_2=n
            }}
            \zeta^*(4n_1+1)
            \zeta^*(4n_2+1)
            +
            2(-4)^{-n}
            \sum_{\substack{
                n_1,n_2\geq 0
                \\
                n_1+n_2=n-1
            }}
            \zeta(4n_1+3)
            \zeta(4n_2+3).
        \end{split}
    \end{equation}
    By Lemma \ref{lem: zeta(I2+1I1+2I)} with $a=1,b=3$, we obtain
    \begin{equation}
        \label{eq: 1,3}
        \zeta^*
        (I_2(\{1,3\}^n))
        +
        \zeta^*
        ({}_1I_1(\{1,3\}^n))
        +
        \zeta^*
        (I_2(\{3,1\}^n))
        \equiv
        8 
        \sum_{\substack{
            n_1,n_2\geq 0 
            \\
            n_1+n_2=n
        }}
        \zeta^*(4n_1+1)
        \zeta^*(4n_2+1)
        \pmod{\pi^2}.
    \end{equation}
    By using \eqref{eq: 1,3}, Proposition \ref{prop: zeta(1I1(1,3))} and \eqref{eq: zeta(I_2(3,1))},
    we have 
    \begin{multline}
        \label{eq: last}
        \zeta^*(I_2(\{1,3\}^n))
        \\
        \begin{aligned}[c]
            &\equiv
            8
            \sum_{\substack{
                n_1,n_2\geq 0
                \\
                n_1+n_2=n
            }}
            \zeta^*(4n_1+1)\zeta^*(4n_2+1)
            +4
            \sum_{\substack{
                n_1,n_2\geq 0
                \\
                n_1+n_2=n
            }}
            \Bigl(
                (-4)^{-n}
                -
                (-4)^{-n_1}
                -
                (-4)^{-n_2}
            \Bigr)
            \zeta^*(4n_1+1)
            \zeta^*(4n_2+1)
            \\
            &\qquad
            -
            2(-4)^{-n}
            \sum_{\substack{
             n_1,n_2\geq 0
             \\
             n_1+n_2=n
            }}
            \zeta^*(4n_1+1)
            \zeta^*(4n_2+1)
            -
            2(-4)^{-n}
            \sum_{\substack{
                n_1,n_2\geq 0
                \\
                n_1+n_2=n-1
            }}
            \zeta(4n_1+3)
            \zeta(4n_2+3)
            \pmod{\pi^2}
            \\
            &=
            2
            \sum_{\substack{
                n_1,n_2\geq 0
                \\
                n_1+n_2=n
            }}
            \bigl(
                (-4)^{-n_1}-2
            \bigr)
            \bigl(
                (-4)^{-n_2}-2
            \bigr)
            \zeta^*(4n_1+1)
            \zeta^*(4n_2+1)
            \\
            &\qquad
            -
            2(-4)^{-n}
            \sum_{\substack{
                n_1,n_2\geq 0
                \\
                n_1+n_2=n-1
            }}
            \zeta(4n_1+3)
            \zeta(4n_2+3).
        \end{aligned}
    \end{multline}
    We remark that \eqref{eq: last} holds when $n=0$.
    Hence, by \eqref{eq:k=(1,3)ver zetaS3=z(I0)+z(I1)t+z(I2)t^2} with 
    \eqref{eq: zeta^*(I_0(1,3))}, \eqref{eq: zeta^*(I_1(1,3))} and \eqref{eq: last},
    we finally get the desired result.
    This completes the proof.
\end{proof}

   %%%%%%%%%%%%        %%%%%%%%%%%%%%%%%%%%%%%%%%%%%%%%%%%%%%%%%%%%

\end{document}